\title{Doubly biased Maker-Breaker Connectivity game}
\author{Dan Hefetz \thanks{School of Mathematical Sciences, Queen Mary University of London,
Mile End Road, London E1 4NS, England. Email: d.hefetz@qmul.ac.uk.}
\and Mirjana Raki\'c \thanks{Department of Mathematics and
Informatics, University of Novi Sad, Serbia. Partly supported by
Ministry of Science and Environmental Protection, Republic of
Serbia. Email: mirjana.rakic@dmi.uns.ac.rs. This paper is a part of
the author's Ph.D. thesis under the supervision of Prof. Milo\v{s}
Stojakovi\'c.} \and Milo\v{s} Stojakovi\'c\thanks{Department of
Mathematics and Informatics, University of Novi Sad, Serbia. Partly
supported by Ministry of Science and Environmental Protection,
Republic of Serbia, and Provincial Secretariat for Science, Province
of Vojvodina. Email: milos.stojakovic@dmi.uns.ac.rs.}}

\documentclass[11pt]{article}
\usepackage{amsmath,amssymb,latexsym,color,epsfig,a4}

\newif\ifnotesw\noteswtrue% T to show comments; F supresses.

\def\({\left(}
\def\){\right)}

\newtheorem{theorem}{Theorem}
\newtheorem{lemma}[theorem]{Lemma}

\newtheorem{corollary}[theorem]{Corollary}

\renewcommand{\epsilon}{\varepsilon}

\newenvironment{proof}{\noindent{\bf Proof\,}}{\hfill$\Box$}

\begin{document}
\maketitle

\newtheorem{smalLem}{Lemma}[theorem]
%\setlength{\marginparwidth}{0 cm} \setlength{\evensidemargin}{0 cm}
%\setlength{\oddsidemargin}{0 cm} \setlength{\textwidth}{16 cm}

%\date{}
%\begin{document}
%\maketitle

\begin{abstract}
In this paper we study the $(a:b)$ Maker-Breaker Connectivity game, played on the edge-set of
the complete graph on $n$ vertices. We determine the winner for almost all values of $a$ and
$b$.
\end{abstract}

\section{Introduction}
Let $X$ be a finite set and let ${\mathcal F} \subseteq 2^X$ be a family of subsets of $X$. In
the biased \emph{$(a:b)$ Maker-Breaker game $(X, {\mathcal F})$} two players, called Maker and
Breaker, take turns in claiming previously unclaimed elements of $X$, with Breaker going
first. In each turn, Breaker claims $b$ elements of $X$, and then Maker claims $a$ elements.
The game ends as soon as every element of $X$ is claimed by either player. Maker wins the game
$(X, {\mathcal F})$ if, by the end of the game, he is able to fully claim some $F \in
{\mathcal F}$; otherwise Breaker wins the game. If Maker has a strategy to win against any
strategy of Breaker, then we say the game is \emph{Maker's win}; otherwise we say it is
\emph{Breaker's win}. The set $X$ will be referred to as the \emph{board} of the game, and the
elements of ${\mathcal F}$ will be referred to as the \emph{winning sets}. The natural numbers
$a$ and $b$ are called the $\emph{bias}$ of Maker and Breaker, respectively.

In this paper, our attention is dedicated to the $(a:b)$
Maker-Breaker \emph{Connectivity game on $E(K_n)$}, that is, the
board of this game is the edge-set of $K_n$, the complete graph on
$n$ vertices, and the winning sets are all spanning connected
subgraphs of $K_n$. From now on we will denote this game by
${\mathcal T}_n$.

%The most natural, and thus also most thoroughly studied, version of an $(X, {\mathcal F})$
%game is the $(1:1)$ game, also known as the \emph{fair} or \emph{unbiased} game.
It is easy to see that the $(1:1)$ game ${\mathcal T}_n$ is Maker's win. In fact, the outcome
of the $(1:1)$ Connectivity game is known even when the board is the edge-set of an arbitrary
graph $G$ -- it was proved in~\cite{Lehman} that this game is Maker's win if and only if $G$
admits two edge disjoint spanning trees. This led Chv\'atal and Erd\H{o}s~\cite{ChEr} to
introduce \emph{biased} games, that is, games for which $(a,b) \neq (1,1)$. Since the $(1:1)$
game ${\mathcal T}_n$ is an easy Maker's win, to give Breaker more power they studied the
biased $(1:b)$ games for $b > 1$.

Chv\'atal and Erd\H{o}s~\cite{ChEr} have observed that Maker-Breaker games are bias monotone,
that is, if some Maker-Breaker $(1:b)$ game $(X, {\mathcal F})$ is Breaker's win, then the
$(1:b+1)$ game $(X, {\mathcal F})$ is Breaker's win as well. Since, unless $\emptyset \in
{\mathcal F}$, the $(1:|X|)$ game $(X, {\mathcal F})$ is clearly Breaker's win, it follows
that, unless $\emptyset \in {\mathcal F}$ or ${\mathcal F} = \emptyset$ (we refer to these
cases as \emph{degenerate}), there exists a unique positive integer $b_0$ such that the
$(1:b)$ game $(X, {\mathcal F})$ is Maker's win if and only if $b \leq b_0$. This value of
$b_0$ is known as the \emph{threshold bias} of the game $(X, {\mathcal F})$. Chv\'atal and
Erd\H{o}s~\cite{ChEr} proved that the threshold bias of ${\mathcal T}_n$ is between $(1/4 -
\varepsilon)n/\ln n$ and $(1 + \varepsilon)n/\ln n$. They have conjectured that the upper
bound is in fact asymptotically best possible. This was verified by Gebauer and
Szab\'o~\cite{GeSa}.

Assume that the $(1:b)$ game ${\mathcal T}_n$ is being played, but instead of playing
optimally (as is always assumed in Game Theory), both players play randomly (they will thus be
referred to as \emph{RandomMaker} and \emph{RandomBreaker}, and the resulting game will be
referred to as the \emph{Random Connectivity game}). It follows that the graph built by
RandomMaker by the end of the game is a random graph $G(n, \lfloor \binom{n}{2}/(b+1) \rfloor)$.
It is well known that almost surely such a graph is connected if $\lfloor \binom{n}{2}/(b+1)
\rfloor \geq (1/2 + \varepsilon) n \ln n$ and disconnected if $\lfloor \binom{n}{2}/(b+1) \rfloor \leq (1/2 - \varepsilon) n \ln n$. Hence, almost surely RandomBreaker wins the game if $b \geq (1 + \tilde{\varepsilon})n/\ln
n$ but loses if $b \leq (1 - \tilde{\varepsilon})n/\ln n$, just like when both players play
optimally. This remarkable relation between positional games and random graphs, first observed
in~\cite{ChEr}, has come to be known as the \emph{probabilistic intuition} or \emph{Erd\H{o}s
paradigm}. Much of the research in the theory of positional games was since then devoted to
finding the threshold bias of certain games and investigating the probabilistic intuition.
Many of these results can be found in~\cite{BeckBook}.

While $(a:b)$ games, where $a > 1$, were studied less than the case $a = 1$, they are not without merit.
Indeed, the small change of going from $a=1$ to $a=2$ has a considerable impact on the outcome and the
course of play of certain positional games (see~\cite{BeckBook}). Moreover, it was shown in~\cite{BRP} that
the \emph{acceleration} of the so-called \emph{diameter-$2$} game partly restores the probabilistic intuition.
Namely, it was observed that, while $G(n,1/2)$ has diameter $2$ almost surely, the $(1:1)$
diameter-$2$ game (that is, the board is $E(K_n)$ and the winning
sets are all spanning subgraphs of $K_n$ with diameter at most $2$)
is Breaker's win. On the other hand, it was proved in~\cite{BRP}
that the seemingly very similar $(2:2)$ game is Maker's win. Further
examples of $(a:b)$ games, where $a > 1$, can be found
in~\cite{BeckBook, BRP, Gebauer}.

Similarly to the $(1:b)$ game, one can define the \emph{generalized threshold bias} for the
$(a:b)$ game as well. Given a non-degenerate Maker-Breaker game $(X, {\mathcal F})$ and $a\geq
1$, let $b_0(a)$ be the unique positive integer such that the $(a:b)$ game $(X, {\mathcal F})$
is Maker's win if and only if $b \leq b_0(a)$. In this paper we wish to estimate $b_0(a)$ for
the Connectivity game ${\mathcal T}_n$ and for every $a$.

Coming back to the Random Connectivity game, its outcome depends on the number of edges
RandomMaker has at the end of the game rather than on the values of $a$ and $b$. Hence, if
$a=a(n)$ and $b=b(n)$ are positive integers satisfying $b \leq a (1 - \varepsilon) n /\ln n$,
for some constant $\varepsilon > 0$, and $b$ is not too large (clearly if for example $b \geq
\binom{n}{2}$, then RandomBreaker wins regardless of the value of $a$), then almost surely
RandomMaker wins the game. Similarly, if $a$ is not too large and $b \geq a (1 + \varepsilon)
n /\ln n$, then almost surely RandomBreaker wins the game. Clearly the outcome of the random
game and of the regular game could vary greatly for large values of $a$ and $b$. For example,
while Breaker wins the $(\binom{n}{2}: n)$ game ${\mathcal T}_n$ in one move, the
corresponding random game is almost surely RandomMaker's win. We prove that for all
``reasonable'' values of $a$ and $b$, the probabilistic intuition is maintained. In the
following we state our results.

\begin{theorem} \label{th::bwin}
Let $\varepsilon > 0$ be a real number and let $n = n(\varepsilon)$
be a sufficiently large positive integer. If $a \leq \ln n$ and $b
\geq (1 + \varepsilon) \frac{an}{\ln(an)}$, then the $(a:b)$ game
${\mathcal T}_n$ is Breaker's win.
\end{theorem}

Note that as $a$ approaches $\ln n$, the lower bound on $b$ in Theorem~\ref{th::bwin} exceeds
$n$ and is therefore trivial. Our next theorem improves that bound for large values of $a$.

\begin{theorem} \label{th:largea}
Let $\varepsilon > 0$ be a real number. If $(1 + \varepsilon) \ln n \leq a < \frac{n}{2e}$ and
$$
b\geq \frac{2a\left(n - 2 + \ln \left\lceil \frac{n}{2a}
\right\rceil\right) + \ln \left\lceil \frac{n}{2a} \right\rceil - 1
+ \frac{2a}{n}}{2a + \ln \left\lceil \frac{n}{2a} \right\rceil - 1 +
\frac{2a}{n}},
$$
then the $(a:b)$ game ${\mathcal T}_n$ is Breaker's win.
\end{theorem}

Finally, for very large values of $a$ we obtain a nontrivial bound on $b$ which suffices to ensure Breaker's win.

\begin{theorem} \label{th:veryLargea}
If $a < \frac{n}{2}$ and $b \geq n-2$, then the $(a:b)$ game
${\mathcal T}_n$ is Breaker's win.
\end{theorem}

For Maker's win we prove the following sufficient conditions, covering the whole range of
possible values of $a$.

\begin{theorem} \label{t::MakersWin}
If  $a = o\left( \sqrt{\frac{n}{\ln n}}\right)$ and
$$
b < \frac{a\left(n-\frac{n}{a \ln n} + \frac{a-1}{2} \ln \left(\frac{n}{a^2 \ln
n}\right)\right)}{\ln n + a + \ln\ln n + 4},
$$
or if $a = \Omega\left( \sqrt{\frac{n}{\ln n}}\right)$,  $a \leq \frac{n-1}{2}$ and $b <
\frac{an}{a+2\ln n - 2\ln a + 4}$, then the $(a:b)$ game ${\mathcal T}_n$ is Maker's win.
\end{theorem}

A straightforward analysis of the results obtained in
Theorems~\ref{th::bwin},~\ref{th:largea},~\ref{th:veryLargea} and~\ref{t::MakersWin}, yields
the following estimates for the generalized threshold bias $b_0(a)$.

\begin{corollary} \label{summary}
\begin{description}
\item [$(i)$] If $a = o(\ln n)$, then $\frac{an}{\ln n} - (1 + o(1)) \frac{an(\ln\ln
n+a)}{\ln^2 n} < b_0(a) < \frac{an}{\ln n} - (1 - o(1)) \frac{an \ln
a}{\ln^2 n}$.
\item [$(ii)$] If $a = c \ln n$ for some $0 < c \leq 1$, then $(1 - o(1)) \frac{cn}{c+1} <
b_0(a) < \min\left\{c n, (1+o(1))\frac{2 n}{3}\right\}$.
\item [$(iii)$] If $a = c \ln n$ for some $c > 1$, then $(1 - o(1)) \frac{cn}{c+1} < b_0(a) < (1 + o(1))\frac{2cn}{2c+1}$.
\item [$(iv)$] If $a = \omega(\ln n)$ and $a = o\left( \sqrt{\frac{n}{\ln n}}\right)$, then $n -
\frac {n\ln n}{a} < b_0(a) < n - (1-o(1))\frac{n\ln \left( n/a\right)} {2a}$.
\item [$(v)$] If $a = \Omega\left( \sqrt{\frac{n}{\ln n}}\right)$ and $a = o(n)$, then $n - (1 + o(1))
\frac{2n\ln \left(n/a\right)}{a} < b_0(a) < n - (1 - o(1)) \frac{n \ln \left(n/a\right)}{2a}$.
\item [$(vi)$] If $a = c n $ for $0 < c < \frac {1}{2e}$, then $n -
\frac{2 \ln (1/c) + 4}{c} < b_0(a) < n-2 - \frac{1-2c}{2c}
\left(\ln(\frac{1}{2 c})-1\right) + o(1)$.
\item [$(vii)$] If $a = c n$ for $\frac{1}{2e} \leq c < \frac {1}{2}$, then $n - \frac{2 \ln (1/c) + 4}{c} < b_0(a) < n-2$.
\end{description}
\end{corollary}

All lower bounds on the threshold bias $b_0(a)$ in Corollary~\ref{summary} are obtained via
Theorem~\ref{t::MakersWin}. The upper bounds are obtained as follows. Theorem~\ref{th::bwin}
is used in $(i)$, Theorem~\ref{th:largea} is used in $(iii)$, $(iv)$, $(v)$ and $(vi)$ and
Theorem~\ref{th:veryLargea} is used in $(vii)$. In $(ii)$, the upper bound of $cn$ is obtained
from Theorem~\ref{th::bwin}, whereas the upper bound $(1+o(1)) \frac{2n}{3}$ is obtained
from $(iii)$ by the bias monotonicity of Maker-Breaker games.

\begin{figure}[htb]
  \centering %
  \epsfig{file=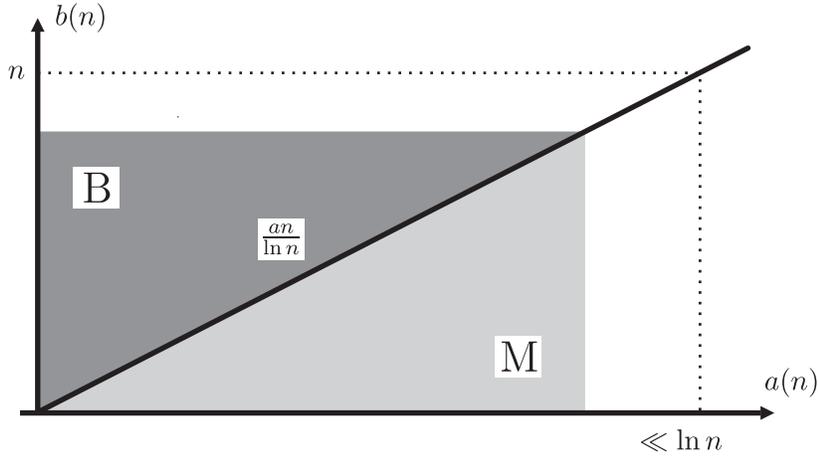, scale=1}
  \caption{Leading term of the threshold bias for $a = o(\ln n)$.\label{f:1}}
  \bigskip
\end{figure}

\begin{figure}[htb]
  \centering %
  \epsfig{file=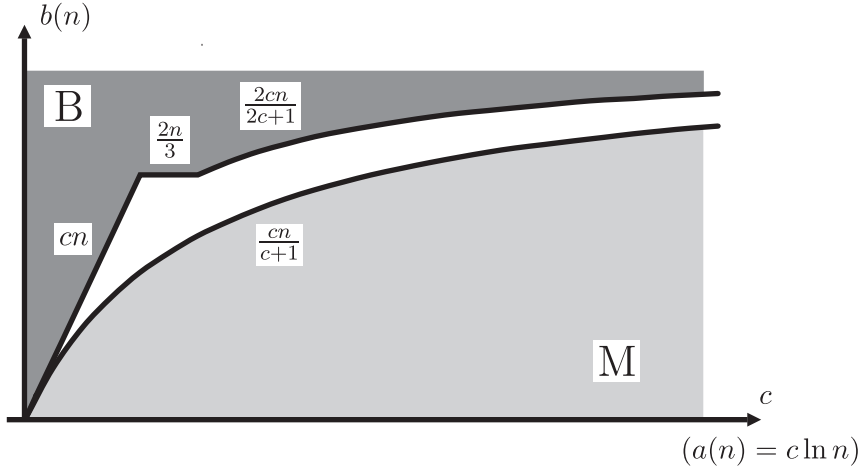, scale=1}
  \caption{Bounds on the threshold bias for $a = c \ln n$, where $c$ is a positive real number.\label{f:2}}
  \bigskip
\end{figure}

Corollary~\ref{summary} gives fairly tight bounds for the threshold bias on the
whole range of the bias $a$. In particular, for $a = o(\ln n)$, the leading term of the
threshold bias is determined exactly, which is depicted in Figure~\ref{f:1}. Then, if $a =
c \ln n$, where $c$ is a positive real number, $(ii)$ and $(iii)$ imply that the threshold bias is
linear in $n$, and the upper bound we obtain is a constant factor away from the lower bound, as shown
in Figure~\ref{f:2}. If $a = \omega(\ln n)$ and $a = o(n)$, it follows by $(iv)$ and $(v)$ that
the leading term of the threshold bias is $n$, and moreover, we obtain upper and lower bounds
for the second order term which are a constant factor away from each other, see
Figure~\ref{f:3}. Finally, for $a = c n$, where $0 < c < 1/2$ is a real number, $(vi)$
and $(vii)$ imply that the threshold bias is just an additive constant away from $n$ as shown in
Figure~\ref{f:4}. For larger values of $a$ we have the trivial upper bound of $b_0(a) \leq n-1$ and
the same lower bound as in $(vii)$ by monotonicity.

\begin{figure}[htb]
  \centering %
  \epsfig{file=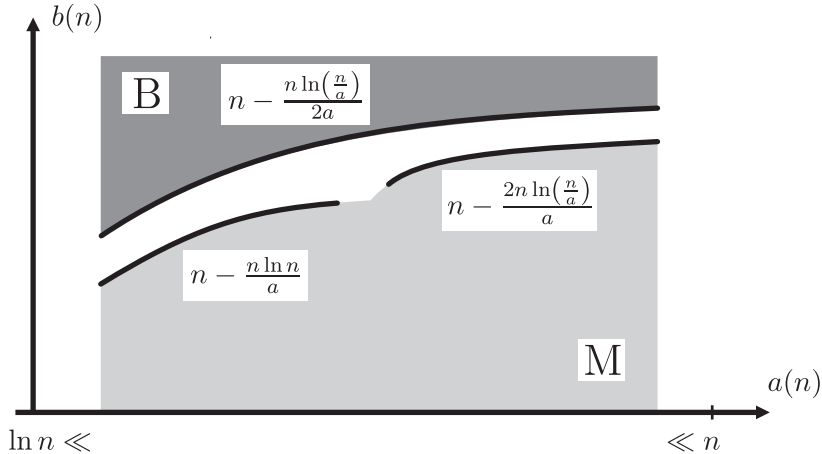, scale=1}
  \caption{Bounds on the threshold bias for $a = \omega(\ln n)$ and $a = o(n)$. \label{f:3}}
  \bigskip
\end{figure}

\begin{figure}[htb]
  \centering %
  \epsfig{file=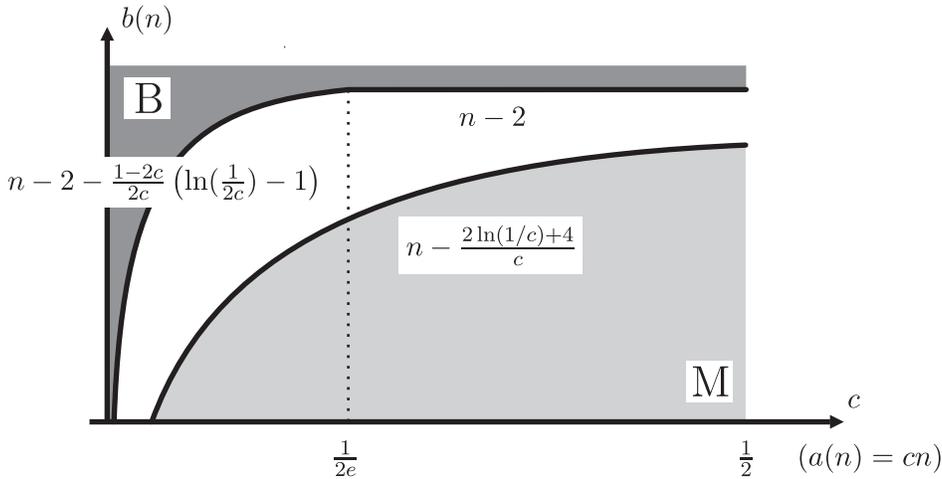, scale=1}
  \caption{Bounds on the threshold bias for $a = c n$, where $0 < c < 1/2$ is a real number.\label{f:4}}
  \bigskip
\end{figure}

For the sake of simplicity and clarity of presentation, we do not make a par\-ti\-cu\-lar
effort to optimize some of the constants obtained in our proofs. We also omit floor and
ceiling signs whenever these are not crucial. Most of our results are asymptotic in nature and
whenever necessary we assume that $n$ is sufficiently large. Throughout the paper, $\ln$
stands for the natural logarithm.

Our graph-theoretic notation is standard and follows that of~\cite{West}. In particular, we
use the following. For a graph $G$, $V(G)$ and $E(G)$ denote its sets of vertices and edges
respectively, $v(G) = |V(G)|$ and $e(G) = |E(G)|$. For a vertex $u \in V(G)$, $d_G(u)$ denotes
the degree of $u$ in $G$. At any point during the game ${\mathcal T}_n$, we denote by $M$
(respectively $B$) the graph which is spanned by the edges Maker (respectively Breaker) has
claimed thus far.

For every positive integer $j$ we denote the \emph{$j$th harmonic number} by $H_j$, that is,
$H_j = \sum_{i=1}^j 1/i$, for every $j \geq 1$. We will make use of the following known fact,
\begin{equation} \label{harmonic}
\ln j + 1/2 <  H_j < \ln j + 2/3 \textrm{ for sufficiently large } j.
\end{equation}

\noindent The rest of the paper is organized as follows. In Section~\ref{box} we analyze the
$(a : b)$ Box Game -- a classical game, introduced in~\cite{ChEr} in order to analyze
Breaker's strategy in Connectivity game. In Section~\ref{sec::bwin} we prove
Theorems~\ref{th::bwin},~\ref{th:largea} and~\ref{th:veryLargea}. In
Section~\ref{sec::MakerStrategy} we prove Theorem~\ref{t::MakersWin}. Finally, in
Section~\ref{sec::openprob} we present some open problems.

\section{$(a:b)$ Box Game} \label{box}
In order to present Breaker's winning strategy for the $(a:b)$ Maker-Breaker Connectivity
game, we first look at the so-called Box Game. The Box Game was first introduced by Chvat\'al
and Erd\H{o}s in~\cite{ChEr}. A hypergraph ${\mathcal H}$ is said to be of type $(k,t)$ if
$|{\mathcal H}| = k$, its hyperedges $e_1, e_2, \ldots, e_k$ are pairwise disjoint, and the
sum of their sizes is $\sum_{i=1}^k |e_i| = t$. Moreover, the hypergraph ${\mathcal H}$ is
said to be \emph{canonical} if $||e_i| - |e_j|| \leq 1$ holds for every $1 \leq i,j \leq k$.
The board of the Box Game $B(k,t,a,b)$ is a canonical hypergraph of type $(k,t)$. This game is
played by two players, called \emph{BoxMaker} and \emph{BoxBreaker}, with BoxMaker having the
first move (in~\cite{ChEr} the first player is actually BoxBreaker, but the version where
BoxMaker is the first player is more suitable for our needs). BoxMaker claims $a$ vertices of
${\mathcal H}$ per move, whereas BoxBreaker claims $b$ vertices of ${\mathcal H}$ per move.
BoxMaker wins the Box Game on ${\mathcal H}$ if he can claim all vertices of some hyperedge of
${\mathcal H}$, otherwise BoxBreaker wins this game.

In~\cite{HLV}, Hamidoune and Las Vergnas have provided, in particular, a sufficient and
necessary condition for BoxMaker's win in the Box Game $B(k,t,a,b)$ for any positive integers
$k,t,a$ and $b$. Unfortunately, this condition can rarely be used in practise. A more
applicable criterion for BoxMaker's win was also provided in~\cite{HLV}, but it turns out not
to be strong enough for certain values of $a$ and $b$. Hence, in this section, we derive a
sufficient condition for BoxMaker's win in the Box Game $B(k,t,a,b)$ which is better suited to
our needs.

Given positive integers $a$ and $b$, we define the following function,
\begin{eqnarray*}
f(k; a, b) := \left\{
\begin{array}{rl}
(k-1)(a+1) \;\;\;\;, & \mbox{if } 1 \leq k \leq b\\
k a \;\;\;\;\;\;\;\;\;\;, & \mbox{if } b < k \leq 2b\\
\left\lfloor \frac{k(f(k-b;a,b)+a-b)}{k-b}\right \rfloor, & \mbox{otherwise}
\end{array}
\right..
\end{eqnarray*}

First we prove the following technical result.

\begin{lemma} \label{lem1}
Let $a,b$ and $k$ be positive integers satisfying $k > b$ and $a - b - 1 \geq 0$, then
\begin{equation} \label{eq::lem1}
f(k;a,b) \geq ka - 1 + \frac {k(a-b-1)}{b} \sum_{j=2}^{\left\lceil k/b
\right\rceil - 1} \frac{1}{j} \enspace .
\end{equation}
\end{lemma}

\begin{proof}
If $b < k \leq 2b$, then the assertion of the lemma holds since $k a \geq k a - 1$.

Otherwise, let $x = \lceil k/b \rceil - 2$. Note that $x$ is the unique positive integer for
which $b < k - x b \leq 2 b$. For every $0 \leq y < x$ we have

\begin{eqnarray}
\frac{k}{k-yb} \cdot f(k-yb;a,b) &=& \frac{k}{k-yb}
\left\lfloor\frac{(k-yb)(f(k-(y+1)b;a,b)+a-b)}{k-(y+1)b}\right\rfloor \nonumber\\
&\geq& \frac{k}{k-yb} \left(\frac{(k-yb)(f(k-(y+1)b;a,b)+a-b)}{k-(y+1)b} - 1\right) \nonumber\\
&=& \frac{k(a-b)}{k-(y+1)b} - \frac{k}{k-yb} + \frac{k}{k-(y+1)b} \cdot f(k-(y+1)b;a,b).
\qquad \label{suma}
\end{eqnarray}

Applying the substitution rule~\eqref{suma} repeatedly for every $0 \leq y < x$ and using the
fact that $\frac{k}{k-xb} \cdot f(k-xb;a,b) = k a$, we obtain

\begin{eqnarray} \label{sumiranje}
f(k;a,b) &\geq& k a - 1 + \sum_{i=1}^{\left\lceil k/b
\right\rceil - 2} \frac{k(a-b)}{k-ib} - \sum_{j=1}^{\left\lceil k/b
\right\rceil - 3} \frac{k}{k-jb} \nonumber\\
&\geq& k a - 1 + k(a-b-1) \sum_{i=1}^{\left\lceil k/b
\right\rceil - 2} \frac{1}{k-ib} \enspace.
\end{eqnarray}

Since $\frac{1}{k-ib} \geq \frac{1}{\left(\lceil k/b \rceil - i \right)b}$
holds for every $1 \leq i \leq \left\lceil k/b \right\rceil - 2$, it follows by~\eqref{sumiranje} that
\begin{eqnarray*}
f(k;a,b) &\geq& k a - 1 + k(a-b-1) \sum_{i=1}^{\left\lceil k/b
\right\rceil - 2} \frac{1}{\left(\lceil k/b \rceil - i \right)b}\\
&=& k a - 1 + k(a-b-1) \sum_{j=2}^{\left\lceil k/b
\right\rceil - 1} \frac{1}{jb}\\
&=& k a - 1 + \frac{k(a-b-1)}{b} \sum_{j=2}^{\left\lceil k/b
\right\rceil - 1} \frac{1}{j} \enspace.
\end{eqnarray*}
\end{proof}

\begin{lemma} \label{lem2}
If $t \leq f(k;a,b) + a$, then BoxMaker has a winning strategy for $B(k,t,a,b)$.
\end{lemma}

\begin{proof}
We prove this lemma by induction on $k$.

If $1 \leq k \leq b$, then $t \leq f(k;a,b) + a = (k-1)(a+1) + a = k(a+1) - 1$.
Since the board is a canonical hypergraph, it follows that there exists a
hyperedge of size at most $a$. In his first move, BoxMaker claims all vertices of such a hyperedge and thus wins.

If $b < k \leq 2b$, then $t \leq f(k;a,b) + a = k a + a$. In his first move, BoxMaker
claims $a$ vertices such that the resulting hypergraph is canonical of type $(k,t')$,
where $t' = t - a \leq k a$. It follows that every hyperedge is of size at most $a$.
Subsequently, in his first move, BoxBreaker claims $b < k$ vertices. Hence, there must
exist an hyperedge which BoxBreaker did not touch in his first move. In his second move,
BoxMaker claims all free vertices of such an hyperedge and thus wins.

Assume then that $k > 2b$ and assume that the assertion of the lemma holds for
every $k_1 < k$, that is, if $t_1 \leq f(k_1;a,b) + a$, then BoxMaker has a
winning strategy for $B(k_1,t_1,a,b)$. In his first move, BoxMaker
claims $a$ vertices such that the resulting hypergraph is canonical of type $(k,t')$,
where $t' = t - a \leq f(k;a,b)$. Subsequently, in his first move, BoxBreaker claims
$b$ board elements. Let $e_1, \ldots, e_{k-b}$ be arbitrary $k-b$ winning sets which BoxBreaker did not touch in his first move. Since BoxMaker's first move results in a canonical hypergraph, it follows that $\hat{t} := \sum_{i=1}^{k-b} |e_i| \leq \frac{k-b}{k} \cdot \left(t' + b \right) \leq \frac{k-b}{k} \cdot t' + b$. Moreover, it follows by the definition of $f$ that $f(k;a,b) \leq \frac{k}{k-b} \left(f(k-b;a,b) + a - b \right)$, implying that $f(k-b;a,b) \geq \frac{k-b}{k} \cdot f(k;a,b) + b - a \geq \hat{t} - a$. Hence, in order to prove that BoxMaker has a winning strategy for $B(k,t,a,b)$, it suffices to prove that BoxMaker has a winning strategy for $B(k-b,\hat{t},a,b)$. This however follows by the induction hypothesis since $k-b < k$ and since, as noted above, $\hat{t} \leq f(k-b;a,b) + a$.
\end{proof}

\section{Breaker's strategy for the $(a:b)$ Connectivity Game} \label{sec::bwin}

\textbf{Proof of Theorem~\ref{th::bwin}} Our proof relies on the approach of Chvat\'al and
Erd\H{o}s~\cite{ChEr}, who proved the special case $a=1$.

For technical reasons we will assume that $\varepsilon < 1/3$. This is allowed by the bias monotonicity of Maker-Breaker games.

Breaker's goal is to isolate some vertex $u \in V(K_n)$ in Maker's
graph. His strategy consists of two phases. In the first phase, he
claims all edges of a clique $C$ on $k := \left\lceil
\frac{an}{(a+1)\ln(an)}\right\rceil$ vertices, such that no vertex
of $C$ is touched by Maker, that is, $d_M(v) = 0$ for every $v \in
C$. In the second phase, he claims all free edges which are incident
to some vertex $v \in C$. If at any point during the game he cannot
follow the proposed strategy, then he forfeits the game.

\textbf{Breaker's strategy:}

\textit{First Phase.} For every $i \geq 1$, just before Breaker's $i$th move,
let $C_i$ denote the largest (breaking ties arbitrarily) clique in Breaker's
graph such that $d_M(v) = 0$ for every $v \in C_i$. Let $\ell_i$ be the largest
integer for which $b_i \leq b$, where $b_i := \binom{a+\ell_i}{2} + (a+\ell_i)|C_i|$.
If $|C_i| \geq k$, then the first phase is over and Breaker proceeds to the
second phase of his strategy. Otherwise, in his $i$th move, Breaker picks
$a+\ell_i$ vertices $u_1^{i}, \ldots, u_{a+\ell_i}^i$ such that $d_M(u_j^i) = 0$
for every $1 \leq j \leq a+\ell_i$, and then claims all edges of $\{(u_{j_1}^i, u_{j_2}^i)
: 1 \leq j_1 < j_2 \leq a+\ell_i\} \cup \{(u_j^i,w) : 1 \leq j \leq a+\ell_i, \; w \in V(C_i)\}$. He then claims additional $b - b_i$ arbitrary edges.

\textit{Second Phase.} Let $C$ be the clique Breaker has
built in the first phase, that is, $|C| \geq k$, $d_M(v) = 0$ holds for every
$v \in V(C)$ and $(u,v) \in E(B)$ holds for every $u,v \in V(C)$. In this
phase Breaker isolates some vertex $v \in V(C)$ in Maker's graph; the game ends as
soon as he achieves this goal (or as soon as $E(B \cup M) = E(K_n)$, whichever
happens first). In order to do so, he restricts his attention to the part of
the board spanned by the free edges of $E_2 := \{(u,v) : u \in V(C), v \in V(K_n)
\setminus V(C)\}$. In order to choose which edges of $E_2$ to claim in each move,
he consults an auxiliary Box Game $B(k,k(n-k),b,a)$ assuming the role of BoxMaker.

Note that if Breaker is able to follow the proposed strategy (without forfeiting the game) and
if BoxMaker has a winning strategy for $B(k,k(n-k),b,a)$, then Breaker wins the $(a:b)$
Connectivity game on $K_n$. Indeed, since $b \geq \binom{a+1}{2} + (a+1)(k-1)$, it follows
that $\ell_i \geq 1$ for as long as $|C_i| < k$. Since Maker can touch at most $a$ vertices of
$C_i$ in his $i$th move, it follows that $|C_{i+1}| \geq |C_i| + \ell_i$, assuming that,
before his $i$th move, there are at least $a + \ell_i$ vertices in $V(K_n \setminus C_i)$
which are isolated in Maker's graph. Moreover, since $|\{(u,v) \in E(K_n) \setminus (E(M) \cup
E(B)) : v \in V(K_n)\}| \leq n-k$ holds for every $u \in V(C)$, it follows that, if BoxMaker
has a winning strategy for $B(k,k(n-k),b,a)$, then Breaker, having built the clique $C$, has a
winning strategy for the $(a:b)$ Connectivity game on $K_n$.

Hence, it suffices to prove that Breaker can follow the proposed strategy without ever having
to forfeit the game and that BoxMaker has a winning strategy for $B(k,k(n-k),b,a)$.

Starting with the former, note that it follows from the definition of $\ell_i$ that,
if $|C_i| \leq \frac{b}{a+3} - \frac{a+2}{2}$, then $\ell_i \geq 3$. Similarly, if
$|C_i| \leq \frac{b}{a+2} - \frac{a+1}{2}$, then $\ell_i \geq 2$ and if $|C_i| \leq
k < \frac{b}{a+1} - \frac{a}{2}$, then $\ell_i \geq 1$. Hence, if Breaker does
not forfeit the game, then his clique reaches size $k$ within at most
$$
\frac{b}{3(a+3)} + \frac{b}{2(a+2)(a+3)} + \frac{b}{(a+1)(a+2)}
$$
moves. Since Maker can touch at most $2a$ vertices in a single move, it follows that during
the first phase of his strategy, the number of vertices which are either in Breaker's clique
or have positive degree in Maker's graph is at most
$$
\frac{b}{3(a+3)}(2a+3)+\frac{b}{2(a+2)(a+3)}(2a+2) + \frac{b}{(a+1)(a+2)}(2a+1) \leq n \enspace,
$$
where this inequality follows since $a \leq \ln n$ and $\varepsilon < 1/3$.

Hence, for as long as $|C_i| < k$, there are vertices of degree $0$ in $M[V(K_n \setminus
C_i)]$ and thus Breaker can follow the proposed strategy without ever having to forfeit the
game.

Finally, since $k > a$ and $b-a-1 \geq 0$, it follows by Lemmas~\ref{lem1} and~\ref{lem2}
that, in order to prove that BoxMaker has a winning strategy for $B(k,k(n-k),b,a)$, it
suffices to prove that
\begin{equation}
k(n-k) \leq kb - 1 + \frac{k(b-a-1)}{a} \sum_{i=2}^{\left\lceil k/a \right\rceil - 1}
\frac{1}{i} + b.
\end{equation}

The latter inequality can be easily verified given our choice of
$k$, the assumed bounds on $a$ and $b$, and by applying
~\eqref{harmonic} on $\sum\limits_{i=2}^{\left\lceil k/a
\right\rceil - 1} \frac{1}{i}$. Here we also use the fact that for
$n\geq 2$ , $ \ln n +1/2 > \ln(n+1)$. {\hfill $\Box$
\medskip\\}

\textbf{Proof of Theorem~\ref{th:largea}} Breaker aims to win Connectivity game on $K_n$ by
isolating a vertex in Maker's graph. While playing this game, Breaker plays (in his mind) an
auxiliary Box Game $B(n,n(n-1),b,2a)$ (assuming the role of BoxMaker). Let $V(K_n) = \{v_1,
\ldots, v_n\}$ and let $e_1, e_2, \ldots, e_n$ be an arbitrary ordering of the winning sets of
$B(n,n(n-1),b,2a)$. In every move, Breaker claims $b$ free edges of $K_n$ according to his
strategy for $B(n,n(n-1),b,2a)$. That is, whenever he is supposed to claim an element of
$e_i$, he claims an arbitrary free edge $(v_i,v_j)$; if no such free edge exists, then he
claims an arbitrary free edge. Whenever, Maker claims an edge $(v_i,v_j)$ for some $1 \leq i <
j \leq n$, Breaker (in his mind) gives BoxBreaker an arbitrary free element of $e_i$ and an
arbitrary free element of $e_j$. Note that every edge of $K_n$ which Maker claims translates
to two board elements of $B(n,n(n-1),b,2a)$. This is why BoxBreaker's bias is set to be $2a$.
It is thus evident that if BoxMaker has a winning strategy for $B(n,n(n-1),b,2a)$, then
Breaker has a winning strategy for the $(a:b)$ Connectivity game on $K_n$.

By Lemma~\ref{lem2}, in order to prove that BoxMaker has a winning strategy for $B(n,n(n-1),b,2a)$, it suffices to prove that $n(n-1) \leq f(n;b,2a) + b$.

Since $b-2a-1 \geq 0$ and $n > 2a$, it follows by Lemma~\ref{lem1},~\eqref{harmonic} and the
fact that for $n\geq 2$, $ \ln n +1/2 > \ln(n+1)$ that
$$
f(n;b,2a) \geq nb - 1 + \frac {n(b-2a-1)}{2a}\left(\ln\left\lceil \frac{n}{2a} \right\rceil -
1\right) \enspace.
$$

Hence, it suffices to prove that
$$
n(n-1) \leq \frac{n(b-2a-1)}{2a} \left(\ln\left\lceil \frac{n}{2a} \right\rceil - 1\right) + n
b - 1 + b \enspace.
$$
It is straightforward to verify that the above inequality holds for
$$b \geq \frac{2a(n-2+ \ln\left\lceil \frac{n}{2a}\right\rceil) +
\ln\left\lceil \frac{n}{2a}\right\rceil - 1 + \frac{2a}{n}}{2a + \ln\left\lceil \frac{n}{2a}
\right\rceil - 1 + \frac{2a}{n}}.$$
 {\hfill $\Box$ \medskip\\}

\textbf{Proof of Theorem~\ref{th:veryLargea}}
In his first move, Breaker claims the edges of some graph of positive minimum degree.
This is easily done as follows. If $n$ is even, then Breaker claims the edges of some
perfect matching of $K_n$ and then he claims additional $b - n/2$ arbitrary free edges. If $n$ is odd, then Breaker claims the edges of a matching of $K_n$ which covers all vertices of
$K_n$ but one, say $u$. He then claims a free edge $(u,x)$ for some $x \in V(K_n)$ and additional $b - (n-1)/2 - 1$ arbitrary free edges.

In his first move, Maker cannot touch all vertices of $K_n$ since $2a < n$.
Let $w \in V(K_n)$ be an isolated vertex in Maker's graph after his first move.
Since $d_B(w) \geq 1$ and $b \geq n-2$, Breaker can claim all free edges which
are incident with $w$ in his second move and thus win.
{\hfill $\Box$ \medskip\\}

\section{Maker's strategy for the $(a:b)$ Connectivity Game} \label{sec::MakerStrategy}

\textbf{Proof of Theorem~\ref{t::MakersWin}} Our proof relies on the approach of Gebauer and
Szab\'o~\cite{GeSa}, who proved the special case $a=1$. First, let us introduce some
terminology. For a vertex $v \in V(K_n)$ let $C(v)$ denote the connected component in Maker's
graph which contains the vertex $v$. A connected component in Maker's graph is said to be
\emph{dangerous} if it contains at most $2b/a$ vertices. We define a \emph{danger function} on
$V(K_n)$ in the same way it was defined in~\cite{GeSa},

\begin{eqnarray*}
{\mathcal D}(v) = \left\{
\begin{array}{rl}
d_B(v), & \mbox{if } C(v)\mbox{ is dangerous}\\
-1, & \mbox{otherwise}
\end{array}
\right..
\end{eqnarray*}

We are now ready to describe Maker's strategy.

\textbf{Maker's strategy:}
Throughout the game, Maker maintains a set $A \subseteq V(K_n)$ of
\textit{active} vertices. Initially, $A = V(K_n)$.

For as long as Maker's graph is not a spanning tree, Maker plays as follows.
For every $i \geq 1$, Maker's $i$th move consists of $a$ steps. For every $1 \leq j
\leq a$, in the $j$th step of his $i$th move Maker chooses an active vertex
$v_i^{(j)}$ whose danger is maximal among all active vertices (breaking ties arbitrarily).
He then claims a free edge $(x,y)$ for arbitrary vertices $x \in C(v_i^{(j)})$
and $y \in V(K_n) \setminus C(v_i^{(j)})$. Subsequently, Maker deactivates $v_i^{(j)}$, that is,
he removes $v_i^{(j)}$ from $A$. If at any point during the game Maker is unable
to follow the proposed strategy, then he forfeits the game.

Note that by Maker's strategy his graph is a forest at any point during the game. Hence, there
are at most $n-1$ steps in the entire game. It follows that the game lasts at most
$\left\lceil \frac{n-1}{a} \right\rceil$ rounds.

In order to prove Theorem~\ref{t::MakersWin}, it clearly suffices to prove that Maker is able
to follow the proposed strategy without ever having to forfeit the game. First, we prove the
following lemma.

\begin{lemma} \label{lem::active}
At any point during the game there is exactly one active vertex in every connected component of Maker's graph.
\end{lemma}

\begin{proof}
Our proof is by induction on the number of steps \textit{r} which Maker
makes throughout the game.

Before the game starts, every vertex of $K_n$ is a connected component
of Maker's graph, and every vertex is active by definition.
Hence, the assertion of the lemma holds for $r=0$.

Let $r \geq 1$ and assume that the assertion of the lemma holds for
every $r' < r$. In the $r$th step, Maker chooses an active vertex $v$ and then
claims an edge $(x,y)$ such that $x \in C(v)$ and $y \notin C(v)$ hold prior
to this move. By the induction hypothesis there is exactly one active vertex
$z \in C(y)$ and $v$ is the sole active vertex in $C(v)$. Since Maker deactivates
$v$ after claiming $(x,y)$, it follows that $z$ is the unique active vertex in
$C(x) \cup C(y)$ after Maker's $r$th step. Clearly, every other component still has
exactly one active vertex.
\end{proof}

We are now ready to prove that Maker can follow his strategy (without forfeiting
the game) for $n-1$ steps. Assume for the sake of contradiction that at some
point during the game Maker chooses an active vertex $v \in C$ and then tries
to connect $C$ with some component of $M \setminus C$, but fails.
It follows that Breaker has already claimed
all the edges of $K_n$ with one endpoint in $C$ and the other in $V \setminus C$. Assume that
Breaker has claimed the last edge of this cut in his $s$th move. As noted above,
$s \leq \left\lceil \frac{n-1}{a} \right\rceil$ must hold. It follows that
$|C| \leq 2b/a$ as otherwise Breaker would have had to claim at least
$\frac{2b}{a} (n-\frac{2b}{a}) > s b$ edges in $s$ moves.
It follows that at any point during the first $s$ rounds of the game there is
always at least one dangerous connected component.

In his $s$th move, Breaker claims at most $b$ edges. Hence,
just before Breaker's $s$th move, $e_B(V(C), V(K_n) \setminus V(C))
\geq |C| \left(n - |C| \right) - b$ must hold.
In particular, $d_B(v) \geq n - \frac{2b}{a} - b$, where $v$ is the unique
active vertex of $C$. Since Maker did not connect $C$ with $M \setminus C$
in his $(s-1)$st move, it follows that, just before this move, there must have
been at least $a+1$ active vertices $v, v_1, \ldots, v_a$ such that the components
$C, C(v_1), \ldots, C(v_a)$ were dangerous and $d_B(u) \geq n - \frac{2b}{a} - b$
for every $u \in \{v, v_1, \ldots, v_a\}$.

For every $1 \leq i \leq s$, let $M_i$ and $B_i$ denote the $i$th
move of Maker and of Breaker, respectively. By Maker's strategy
$v_1^{(1)}, \ldots, v_1^{(a)}, v_2^{(1)}, \ldots, v_2^{(a)}, \ldots,
v_{s-1}^{(1)}, \ldots, v_{s-1}^{(a)}$ are of maximum degree in
Breaker's graph among all active vertices at the appropriate time,
that is, just before Maker's $j$th step of his $i$th move,
$d_B(v_i^{(j)})$ is maximal among all active vertices. Let $v_s$ be
an active vertex of maximum degree in Breaker's graph just before
Maker's $s$th move. Note that, for every $u \in \{v_1^{(1)}, \ldots,
v_1^{(a)}, v_2^{(1)}, \ldots, v_2^{(a)}, \ldots, v_{s-1}^{(1)},
\ldots, v_{s-1}^{(a)}, v_s\}$, if $u$ is active, then $C(u)$ is a
dangerous component. For every $1 \leq i \leq s-1$, let $A_{s-i} =
\{v_{s-i}^{(1)}, \ldots, v_{s-i}^{(a)}, \ldots,v_{s-1}^{(1)},
\ldots, v_{s-1}^{(a)},v_s\}$ denote the subset of vertices of
$\{v_1^{(1)}, \ldots, v_1^{(a)}, v_2^{(1)}, \ldots,  v_2^{(a)},
\ldots, v_{s-1}^{(1)}, \ldots, v_{s-1}^{(a)}, v_s\}$ that are still
active just before Maker's $(s-i)$th move and let $A_s = \{v_s\}$.
For every $A \subseteq V$, let $\overline {{\mathcal D}}_{B_i}(A) =
\frac{\sum_{v \in A} {\mathcal D}(v)}{|A|}$ denote the average
danger value of the vertices of $A$, immediately before Breaker's
move $B_i$. The average danger $\overline{{\mathcal D}}_{M_i}(A)$ is
defined analogously.

Since Maker always deactivates vertices of maximum danger, thus
reducing the average danger value of active vertices, we have the following
lemma.
\begin{lemma} \label{s2}
$\overline{{\mathcal D}}_{M_{s-i}}(A_{s-i}) \geq \overline{{\mathcal
D}}_{B_{s-i+1}}(A_{s-i+1})$ holds for every $1 \leq i \leq s-1$.
\end{lemma}

\begin{proof}
Let $u \in A_{s-i+1} = \{v_{s-i+1}^{(1)}, \dots, v_{s-i+1}^{(a)}, \dots, v_{s-1}^{(1)}, \dots, v_{s-1}^{(a)}, v_s\}$ be an arbitrary vertex. Since $C(u)$ is a dangerous component immediately before $B_{s-i+1}$, the danger ${\mathcal D}(u)$ does not change during $M_{s-i}$.
It follows that $\overline{{\mathcal D}}_{M_{s-i}}(A_{s-i+1})= \overline{{\mathcal
D}}_{B_{s-i+1}}(A_{s-i+1})$.

Note that the vertices contained in
$A_{s-i+1}$ were still active before $M_{s-i}$.

Following his strategy, Maker deactivated all the vertices of $\{v_{s-i}^{(1)},
\ldots, v_{s-i}^{(a)}\}$, because their danger values were the largest among all
the active vertices of $\{v_{s-i}^{(1)}, \ldots, v_{s-i}^{(a)}, \ldots,
v_{s-1}^{(1)}, \ldots, v_{s-1}^{(a)}, v_s\}$. It follows that
$$
\min \{{\mathcal D}(v_{s-i}^{(1)}), \ldots, {\mathcal D}(v_{s-i}^{(a)})\}
\geq \max\{{\mathcal D}(v_{s-i+1}^{(1)}), \ldots, {\mathcal D}(v_{s-i+1}^{(a)}),
\ldots, {\mathcal D}(v_{s-1}^{(1)}), \ldots, {\mathcal D}(v_{s-1}^{(a)}), {\mathcal D}(v_s)\},
$$
and thus $\overline{{\mathcal D}}_{M_{s-i}}(A_{s-i}) \geq \overline{{\mathcal D}}_{M_{s-i}}(A_{s-i+1})$, as claimed.
\end{proof}

\medskip

The following lemma gives two estimates on the change of the danger
value caused by Breaker's moves.
\begin{lemma} \label{s3}
The following two inequalities hold for every $1 \leq i \leq s-1$.
\begin{itemize}
\item [$(i)$] $\overline{{\mathcal D}}_{M_{s-i}}(A_{s-i})-
\overline{{\mathcal D}}_{B_{s-i}}(A_{s-i}) \leq \frac{2 b}{ai+1} < \frac{2 b}{ai}$.
\item [$(ii)$] Define a function $g : \{1, \ldots, s\} \rightarrow \mathbb{N}$ by setting $g(i)$ to be the number of edges with both endpoints in $A_i$ which Breaker has claimed
during the first $i-1$ moves of the game. Then,
$$
\overline{{\mathcal D}}_{M_{s-i}}(A_{s-i})- \overline{{\mathcal D}}_{B_{s-i}}(A_{s-i}) \leq
\frac{b + a^2(i-1) + a + {a \choose 2} + g(s-i+1) - g(s-i)}{ai+1}.
$$
\end{itemize}
\end{lemma}

\begin{proof}
\begin{itemize}
\item [$(i)$] The components
$C(v_{s-i}^{(1)}), \ldots, C(v_{s-i}^{(a)}), \ldots, C(v_{s-1}^{(1)}), \ldots, C(v_{s-1}^{(a)}), C(v_s)$ are dangerous before Maker's $(s-i)$th move.
During Breaker's moves, the components of Maker's graph do not change.
Hence, the change of the danger values of the vertices of $A_{s-i}$, caused by Breaker's $(s-i)$th move, depend solely on the change of their degrees in Breaker's graph. In his $(s-i)$th move, Breaker claims $b$ edges and thus the increase of the sum of the
degrees of the vertices of $\{v_{s-i}^{(1)}, \ldots, v_{s-i}^{(a)}, \ldots, v_{s-1}^{(1)}, \ldots, v_{s-1}^{(a)}, v_s\}$ is at most $2b$. The size of $A_{s-i}$ is $ai+1$. Thus
$\overline{{\mathcal D}}_{B_{s-i}}(A_{s-i})$ increases by at most $\frac {2b}{ai+1}$ during $B_{s-i}$.

\item [$(ii)$] Let $p$ denote the number of edges $(x,y)$ claimed by Breaker during $B_{s-i}$ such that $\{x,y\} \subseteq A_{s-i}$ and let $q = b - p$. Hence, the increase of the sum $\sum_{u \in A_{s-i}} d_B(u)$ during $B_{s-i}$ is at most $2p+q = p+b$.
It follows that $\overline{{\mathcal D}}_{B_{s-i}}(A_{s-i})$ increases by at most $\frac{b + p}{ai+1}$ during $B_{s-i}$. It remains to prove that $p \leq a^2(i-1) + a + {a \choose 2} + g(s-i+1) - g(s-i)$.

During his first $s-i-1$ moves, Breaker has claimed exactly $g(s-i)$ edges with
both their endpoints in $A_{s-i}$. Hence, during his first $s-i$ moves, Breaker has claimed exactly $g(s-i) + p$ edges with both their endpoints in $A_{s-i}$. Exactly $g(s-i+1)$ of these edges have both their endpoints in $A_{s-i+1} = A_{s-i} \setminus \{v_{s-i}^{(1)}, \ldots, v_{s-i}^{(a)}\}$. There can be at most $a \choose 2$ edges connecting two vertices of
$\{v_{s-i}^{(1)}, \ldots, v_{s-i}^{(a)}\}$. Moreover, each vertex of $\{v_{s-i}^{(1)}, \ldots, v_{s-i}^{(a)}\}$ is adjacent to at most $a(i-1)+1$ vertices of $A_{s-i+1}$. Combining all of these observations, we conclude that
$g(s-i) + p \leq g(s-i+1) + a^2 (i-1) + a + {a \choose 2}$, entailing $p \leq {a \choose 2} + a^2(i-1) + a + g(s-i+1) - g(s-i)$ as claimed.
\end{itemize}
\end{proof}

Clearly, before the game starts ${\mathcal D}(u) = d_B(u) = 0$ holds for every vertex $u$.
In particular $\overline{{\mathcal D}}_{B_1}(A_{1}) = 0$. Using our assumption that Breaker wins the game, we will obtain a contradiction by showing that $\overline{{\mathcal D}}_{B_1}(A_{1}) > 0$.

Note that, as previously observed, $\overline{{\mathcal D}}_{B_s}(A_s) \geq n - \frac{2b}{a} - b$. We will use this fact, Lemma~\ref{s2}, Lemma~\ref{s3}, the inequalities $\frac{1}{ai+1} < \frac{1}{ai}$ and $b + a^2(i-1) + a + {a \choose 2} + g(s-i+1) - g(s-i) \geq 0$ (which hold for every $1 \leq i \leq s-1$), and~\eqref{harmonic}, in order to reach the aforementioned contradiction.

Let $k := \left\lfloor \frac{n}{a^2 \ln n}\right\rfloor$.

First, assume that $a = o(\sqrt{n/\ln n})$. We split the game into two parts -- the main game and the last $k$ moves. In these last moves, we will use a more delicate estimate on the effect of Breaker's move on the average danger. We distinguish between two cases.

\noindent\textbf{Case 1:} $s < k$.

\begin{align*}
\overline{{\mathcal D}}_{B_1}(A_{1}) &= \overline{{\mathcal D}}_{B_s}(A_s)
+\sum_{i=1}^{s-1}\left(\overline{{\mathcal D}}_{M_{s-i}}(A_{s-i}) - \overline{{\mathcal
D}}_{B_{s-i+1}}(A_{s-i+1})\right)\\ \displaybreak[0]
&  - \sum_{i=1}^{s-1}
\left(\overline{{\mathcal D}}_{M_{s-i}}(A_{s-i}) -
\overline{{\mathcal D}}_{B_{s-i}}(A_{s-i})\right)\\
&\geq n- \frac{2b}{a} - b + \sum_{i=1}^{s-1}0-\sum_{i=1}^{s-1} \frac{b + a^2(i-1)
+ a + {a \choose 2} + g(s-i+1) - g(s-i)}{ai}\\
&\geq n - \frac{b}{a} (H_{s-1}+2+a) - a(s-1) + a H_{s-1} - H_{s-1} - \frac{a-1}{2} H_{s-1} \\
\displaybreak[0]
&  - \frac{g(s)}{a} + \sum_{i=1}^{s-2}\frac{g(s-i)}{ai(i+1)} + \frac{g(1)}{a (s-1)}\\
&\geq n - \frac{b}{a} (H_{s-1}+2+ a) + \frac{a-1}{2} H_{s-1} - a(s-1) \\
& \qquad \qquad \qquad \qquad \qquad(\mbox{since } g(s)=0 \mbox{ and
} g(s-i) \geq 0)\\ \displaybreak[0] &> n - \frac{b}{a} (H_s + 2 + a)
+ \frac{a-1}{2} H_s - a s\\ \displaybreak[0]
&> n - \frac{b}{a} (\ln k + 3 + a) + \frac{a-1}{2} \ln k - a k \\
&  \qquad \qquad \qquad \qquad \qquad (\mbox{since } s < k)\\
\displaybreak[0] &> n - \frac{n-\frac{n}{a \ln n} + \frac{a-1}{2}
\cdot \ln \left(\frac{n}{a^2\ln
n}\right)}{\ln n + a + \ln\ln n + 4} \left(\ln n - \ln a - \ln\ln n + a + 3 \right) \\
\displaybreak[0]
&  - \frac{n}{a \ln n} + \frac{a-1}{2} \cdot \ln \left(\frac{n}{a^2 \ln n}\right)\\
%&\geq O\left(\frac {n(\ln\ln n+\ln a)}{(3c+1)\ln n}\right)-O\left(\frac n{\ln n}\right),[c>0,\mbox{constant depending on } a]\\
&> 0.
\end{align*}

\textbf{Case 2:} $s \geq k$.
\begin{align*}
\overline{{\mathcal D}}_{B_1}(A_{1})&= \overline{{\mathcal D}}_{B_s}(A_s) +
\sum_{i=1}^{s-1}\left(\overline{{\mathcal D}}_{M_{s-i}}(A_{s-i})-\overline{{\mathcal D}}_{B_{s-i+1}}(A_{s-i+1})\right)\\
\displaybreak[0]
&  - \sum_{i=1}^{k-1}\left(\overline{{\mathcal D}}_{M_{s-i}}(A_{s-i}) - \overline{{\mathcal
D}}_{B_{s-i}}(A_{s-i})\right)-\sum_{i=k}^{s-1}
\left(\overline{{\mathcal D}}_{M_{s-i}}(A_{s-i})-\overline{{\mathcal D}}_{B_{s-i}}(A_{s-i}) \right)\\ \displaybreak[0]
&\geq n - \frac{2b}{a} - b + \sum_{i=1}^{s-1} 0 \\
&  - \sum_{i=1}^{k-1} \frac{b + a^2(i-1) + a + {a \choose 2} + g(s-i+1) - g(s-i)}{ai} - \sum_{i=k}^{s-1} \frac{2 b}{ai}\\ \displaybreak[0]
&\geq n - \frac{2b}{a} - b - \frac{b}{a} H_{k-1} + \frac{a-1}{2} H_{k-1} - a(k-1)\\
&  -\frac{g(s)}{a} + \sum_{i=1}^{k-2} \frac{g(s-i)}{a i(i+1)} + \frac{g(s-k+1)}{a (k-1)} - \frac{2 b}{a}(H_{s-1}-H_{k-1})\\ \displaybreak[0]
&\geq n - \frac{b}{a} (2 H_s - H_{k-1} + 2 + a) + \frac{a-1}{2} H_{k -1} - a (k-1)\\
&  \qquad \qquad \qquad \qquad \qquad(\mbox{since } g(s)=0 \mbox{
and } g(s-i) \geq 0)\\ \displaybreak[0] &> n - \frac{b}{a} (2\ln
\left(\frac{n-1}{a}\right) - \ln k + a + 4) +
\frac{a-1}{2} \cdot \ln k - a k\\
&\geq n - \frac{b}{a} (\ln n + \ln\ln n + a + 4) + \frac{a-1}{2}
\cdot \ln \left(\frac{n}{a^2 \ln n}\right) - \frac n{a \ln n}\\ \displaybreak[0]
&> n - \frac {n-\frac n{a \ln n} + \frac{a-1}{2} \cdot \ln \left(\frac{n}{a^2 \ln n}\right)}
{\ln n + a + \ln \ln n + 4}(\ln n + \ln\ln n + 4 + a)\\ \displaybreak[0]
&  +\frac{a-1}{2} \cdot \ln \left(\frac{n}{a^2 \ln n}\right) - \frac{n}{a \ln n}\\
%&\geq&n-\frac {(\ln n+a)(1+o(1))}{\ln n+a}n(1-o(1))-\frac a 2 \ln n(1-o(1))-\frac n{\ln n}\\
&= 0 \,.
\end{align*}

Next, assume that $a = \Omega(\sqrt{n/\ln n})$ and $a \leq \frac{n-1}{2}$.
In this case, the game does not last long.

\begin{eqnarray*}
\overline{{\mathcal D}}_{B_1}(A_{1}) &=& \overline{{\mathcal D}}_{B_s}(A_s) +
\sum_{i=1}^{s-1}\left(\overline{{\mathcal D}}_{M_{s-i}}(A_{s-i}) - \overline{{\mathcal
D}}_{B_{s-i+1}}(A_{s-i+1})\right) \\
& & - \sum_{i=1}^{s-1}\left(\overline{{\mathcal D}}_{M_{s-i}}(A_{s-i}) - \overline{{\mathcal D}}_{B_{s-i}}(A_{s-i})\right)\\
&>& n - \frac{2 b}{a} - b + \sum_{i=1}^{s-1} 0 - \sum_{i=1}^{s-1}\frac{2 b}{ai} \\
&>& n- \frac{2 b}{a} - b - \frac{2 b}{a} (\ln s + 1) \\
&=& n - \frac{2b}{a} (2 + \frac{a}{2} + \ln s) \\
&\geq& n - \frac{2 b}{a} \left(2 + \frac{a}{2} + \ln \left(\frac{n-1}{a}\right)\right) \\
&>& n- \frac{n}{\ln n - \ln a + 2 + \frac{a}{2}} \left(2 + \frac{a}{2} + \ln
\left(\frac {n-1}{a}\right)\right)\\
&\geq& 0 \,.
\end{eqnarray*}
{\hfill $\Box$ \medskip\\}

\section{Concluding remarks and open problems} \label{sec::openprob}

\paragraph{Determining the threshold bias.} In this paper we have tried to determine the
winner of the $(a : b)$ Connectivity game on $E(K_n)$ for all values of $a$ and $b$. We have
established lower and upper bounds on the threshold bias $b_0(a)$ for every value of $a$. For
most values, these bounds are quite sharp. However, for $a = c \ln n$, where $c > 0$ is fixed,
the first order terms in the upper bound and the lower bound differ. For that reason, we feel
that an improvement of the bounds in this case would be particularly interesting.

\paragraph{Analyzing other games.} There are many well-studied Maker-Breaker games played on the
edge-set of the complete graph for which, in the biased $(a : b)$ version, the identity of the
winner is known for $a=1$ and almost all values of $b$. Some examples are the Hamilton cycle
game (see~\cite{ChEr} and~\cite{Krivelevich}) and the $H$-game, where $H$ is some fixed graph
(see~\cite{BL}). It would be interesting to analyze these games for other values of $a$ (and
corresponding values of $b$) as well.

We note that all the results of the present paper also hold for the Positive Minimum Degree game,
where the Maker's goal is to touch all $n$ vertices of the board $K_n$, and Breaker's goal
is to prevent Maker from doing this. Indeed, if Maker wins the Connectivity game, then he clearly
wins the Positive Minimum Degree game with the same parameters as well. On the other hand, in
all our results that guarantee Breaker's win in the Connectivity game, we in fact prove that Breaker can
isolate a vertex in Maker's graph, which clearly also ensures Breaker's win in the Positive Minimum
Degree game.

\end{document}